\documentclass[11pt,notitlepage,a4paper,reqno]{amsart}
\usepackage[usenames]{color}
\usepackage{colortbl}
\usepackage{amssymb,amsfonts, amsmath}
\usepackage[latin1]{inputenc}
\newtheorem{theorem}{Theorem}[section]

\newtheorem{lemma}[theorem]{Lemma}
\newtheorem{proposition}[theorem]{Proposition}
\theoremstyle{definition}
\newtheorem{definition}[theorem]{Definition}
\newtheorem{remark}[theorem]{Remark}
\theoremstyle{plain}

\def\R{\mathbb{R}}
\def \Bdbeta{\boldsymbol{\beta}}
\def \Bdsigma{\boldsymbol{\sigma}}

\numberwithin{equation}{section} \makeatletter
\newcommand{\cvd}{\thinspace\null\nobreak\hfill
\hbox{\vbox{\kern-.2pt\hrule height.2pt
depth.2pt\kern-.2pt\kern-.2pt \hbox to1.8mm{\kern-.2pt\vrule width.4pt
\kern-.2pt\raise1.8mm\vbox to.2pt{}\lower0pt\vtop to.2pt{}\hfil\kern-.2pt
\vrule width.4pt\kern-.2pt}\kern-.2pt\kern-.2pt\hrule height.2pt depth.2pt
\kern-.2pt}}\par\medbreak}

\begin{document}
	\title[Local boundedness for solutions
	of anisotropic problems]{{Local boundedness for solutions
			of a class of non-uniformly elliptic anisotropic problems
	}}
	
	\author{Stefano Biagi, Giovanni Cupini, Elvira Mascolo}
	

	
	\thanks{The authors are members of the
Gruppo Nazionale per l'Analisi Matematica, la Proba\-bilità e le loro
Applicazioni (GNAMPA) of the Istituto Nazionale di Alta Matematica
(INdAM). S. Biagi is partially 
supported by the PRIN 2022 project 2022R537CS \emph{$NO^3$ - Nodal Optimization, NOnlinear elliptic equations, NOnlocal geometric problems, with a focus on regularity}, founded by the European Union - Next Generation EU., and also by the Indam-GNAMPA project CUP E5324001950001 - {\em Problemi singolari e degeneri: esistenza, unicità e analisi delle proprietà  qualitative delle soluzioni}.
G. Cupini acknow\-ledges
financial support under the National Recovery and Resilience Plan
(NRRP), Mission 4, Component 2, Investment 1.1, Call for tender No.
104 published on 2.2.2022 by the Italian Ministry of University and
Research (MUR), funded by the European Union - NextGenerationEU -
Project PRIN\_CITTI 2022 - Title ``Regularity problems in sub-Riemannian
structures'' - CUP J53D23003760006 - Bando 2022 - Prot. 2022F4F2LH, 
 and also through the INdAM-GNAMPA Project CUP E53C23001670001
 - {\em Interazione ottimale tra la regolarità dei coefficienti e l'anisotropia del problema in funzionali integrali a crescite non standard}.
}
		\begin{abstract}
	We  consider a class of {energy integrals}, associated to nonli\-near and 
	non-uniformly elliptic equations, with integrands $f(x,u,\xi)$
	sati\-sfying anisotropic $p_i,q$-growth conditions of the form
	$$ 
	\sum_{i=1}^n \lambda_i (x)|\xi_i|^{p_i}\le {f}(x,u,\xi)\le \mu (x)\left\{|\xi|^{q} +
	|u|^{\gamma}+1\right\}
	$$
		for some exponents $\gamma\ge q\geq p_i>1$, and non-negative
		functions $\lambda_i,\mu$ 
	subject to suitable summability assumptions. We prove the local boundedness of scalar local quasi-minimizers of such integrals.
	\end{abstract}
	\maketitle
	
\begin{center}
	\dedicatory{\footnotesize{\textit{This paper is dedicated to Gioconda Moscariello  in the
			occasion of her 70th birthday, with friendship and   esteem for her influential contributions on the areas of  Partial Differential Equations and of Calculus of Variations.}}}
\end{center}	

\section{Introduction}
Given a local minimizer $u$ of an integral functional, it is important to investigate the additional structure and integrability properties of the associa\-ted energy density in order to establish the regularity of $u$, at least in the interior of the domain $\Omega$.

A powerful method in regularity theory is due to De Giorgi \cite{degiorgi}, who in\-troduced an approach based on \emph{Caccioppoli-type inequalities on level sets}. This method, through an iteration procedure, first yields local boundedness and then, by a delicate and ingenious argument, establishes the local H\"older continuity of weak solutions to the linear uniformly elliptic equation:
\begin{equation}\label{e:deg}
 \sum_{i,j=1}^{n}\frac{\partial }{\partial x_{i}}\left(
 a_{ij}\left( x\right) \,u_{x_{j}}\right) = 0 \quad \text{in }
 \Omega \subset \mathbb{R}^{n}, \quad n \geq 2,
\end{equation}
where the coefficients $a_{ij}$ are measurable and symmetric functions satisfying the uniform ellipticity condition 
$$\lambda \le a_{ij}(x) \le \mu,$$ 
with $\lambda, \mu$ positive constants such that $\lambda \leq \mu$.

In 1971, Trudinger \cite{Trudinger} considered the same equation \eqref{e:deg}, but with mea\-surable coefficients $a_{ij}$ which satisfy
$$
\lambda(x)|\xi |^{2}\leq  a_{ij}(x)\,\xi_i\xi_j\leq \mu (x)|\xi|^2
$$
without assuming that \( \lambda^{-1}, \mu \in L^\infty \), that is, 
\emph{without the uniform ellipticity condition}. Trudinger proved that, if
$$
\lambda^{-1} \in L_{\mathrm{loc}}^{r}(\Omega), \quad \lambda^{-1} \mu^2 \in L_{\mathrm{loc}}^{s}(\Omega),
$$
for suitable exponents $r,s\geq 1$ satisfying the relation
\begin{equation}\label{e:Trud}\frac{1}{r} + \frac{1}{s} < \frac{2}{n},\end{equation}  
then the weak solutions are locally bounded. For related results, see \cite{BiCuMa, mosc-car-pass1, car-pass2006, CupMarMas17, CupMarMas18,  fabes-kenig-serapioni, mosc1}. Trudinger's  result has been recently improved by Bella and Sch\"affer \cite{bella-sch}, who relaxed the above relation, proving that the right-hand side in \eqref{e:Trud} can be replaced by $\frac{2}{n-1}$.
\medskip

In this paper, we consider integral functionals associated with non-unifor\-mly elliptic problems. More precisely, we investigate functionals from the Calculus of Variations of the form:
\begin{equation} \label{funzionale}
	\mathcal{F}:W^{1,1}_{\mathrm{loc}}(\Omega)\rightarrow [0,\infty], \qquad
	\mathcal{F}(u) := \int_{\Omega}f(x,u,D u)\,\mathrm{d}x,
\end{equation}
where $\Omega \subseteq \mathbb{R}^n$ is a bounded open set and $f$ is a real-valued Carath\'eodory function on $\Omega \times \mathbb{R} \times \mathbb{R}^n$, satisfying the following \emph{anisotropic} and \emph{non-uniform} growth conditions:

\begin{center}
\emph{For some exponents $1 < p_i \le q \le \gamma$, and for measurable functions 
$\lambda_i, \mu: \Omega \rightarrow [0,\infty)$, with 
$i = 1,\ldots,n$, we assume
\begin{equation}\label{eq:crescita}
	\sum_{i=1}^{n}\lambda_{i}(x)|\xi_{i}|^{p_{i}}\le f(x,u, \xi)\le\mu(x)\left(|\xi|^{q} +
		|u|^{\gamma}+1\right)
\end{equation}
for almost every $x \in \Omega$ and every $\xi = (\xi_1,\ldots,\xi_n) \in \mathbb{R}^n$, where
$$
\lambda_{i}^{-1} \in L_{\mathrm{loc}}^{r_{i}}(\Omega), \quad \mu \in L_{\mathrm{loc}}^{s}(\Omega),
$$
for some $r_{i} \in [1,\infty]$ and $s \in (1,\infty]$.}
\end{center}
Under suitable relationships among the exponents (see \eqref{eq:assumptionsParameters}), we prove that local-minimizers (or, more generally,  \emph{quasi-minimizers}) $u$ of \eqref{funzionale} are locally bounded. Moreover, an explicit e\-sti\-mate for $\|u\|_{L^\infty}$ on a given compact subset of $\Omega$ is provided (see \eqref{stima-teorema}). This result extends a previous work by the same authors, non-isotropic case where the case $p = p_i = q$ was considered, see \cite{BiCuMa}. In this setting of anisotropic growth and degenerate ellipticity conditions, we cite the very recent articles \cite{FePaPo} (scalar framework) and  \cite{amb-cup-mas, gao} (vectorial),  where local boundedness results have been proved. 

It is worth mentioning that, 
since the  counterxamples in \cite{gia3,mar87},  the  foundational works by Marcellini~\cite{mar89, mar91} and the  papers  \cite{bocmarsbo, cupmarmas,fussbo1,fussbo2,str},  it is now well-known that 
\emph{the regularity of the minimizers may be lost, even in the non-degenerate case} (i.e., when the coefficients   $\lambda_i$'s are strictly positive), 
if the exponents  $p_i$'s and $q$ are spread out  
\vspace{0.1cm}

From a physical perspective, anisotropy refers to the property of a material to exhibit different mechanical, electrical, thermal, or optical behaviors depending on the direction. This feature arises in many physical models, including magnetic theory (via density functionals), composite materials, convection-diffusion problems, and the analysis of sedimentary rocks. 
  Since the aforementioned papers \cite{mar89, mar91}, anisotropic problems have received considerable attention, and it is nearly impossible to provide an exhaustive list of references. We mention here the articles \cite{BLPV,  mosc-car-pass2, mosc-car-pass3, cupmarmas13}; see also the recent surveys~\cite{mar20-2, MR} for an overview of the subject.
\medskip

\emph{Plan of the paper}. The structure of the paper is as follows. In Section~2, we introduce the notation used throughout the work. Section~3 contains the precise assumptions on the integrand and the statement of the main regularity result. Section~4 is devoted to the proof of a new anisotropic Caccioppoli inequality, which is a key tool in the analysis. Finally, in Section~5, we develop the iteration scheme that leads to the local boundedness of minimizers and establish the \( L^\infty \) estimate.

\section{Notation}
  Throughout the paper, we will tacitly
exploit all the notation listed below; we thus refer the Reader to this list
for any (non-standard) notation encountered.
\begin{itemize}
 
 \item If $A,B\subseteq\R^n$ are arbitrary sets, we write $A\Subset B$
 if $A$ is \emph{relatively compact in $B$}, that is,
 $A$ is bounded and $\overline{A}\subset B$;
 \item Given any $n\in\mathbb{N}$ and any $1\leq \beta < n$, we indicate by $\beta^*$ the 
 \emph{So\-bo\-lev exponent} of $\beta$ with respect to $\R^n$, that is,
 \begin{equation}\label{p*}
 \beta^*:=\frac{n\beta}{n-\beta}.
\end{equation}
 \item Given any $\beta\in[1,\infty]$, we indicate by $\beta'$ the \emph{conjugate exponent}
 of $\beta$ in the classical H\"older inequality, that is,
 $$\beta' = \begin{cases}
 \frac{\beta}{\beta-1}&\text{if $1<\beta<\infty$} \\
 \infty & \text{if $\beta = 1$} \\
 1 & \text{if $\beta = \infty$}
 \end{cases}$$
 \item Given any multi-index $\boldsymbol{\beta} = (\beta_1,\ldots,\beta_n)$ with $1\leq\beta_1,\ldots,\beta_n\leq\infty$, we indicate
 by $\overline{\boldsymbol{\beta}}$ the harmonic average of
 the $\beta_i$'s, that is,
 \begin{equation}\label{definizione-media}
 \boldsymbol{\overline{\beta}} = \frac{n}{\sum_{i = 1}^n1/\beta_i}
\end{equation}
 (with the usual convention $1/\infty = 0$).
\end{itemize}

\section{Assumptions and statement of the main results}
\label{scalar-assumptions}

Consider the integral functional 
\begin{equation} \label{functional}
F(u;\Omega):=\int_{\Omega }f(x,u,Du)\,dx,
\end{equation}
where $\Omega$ is an open and  bounded subset of $\mathbb{R}^n$ (with $n\ge
2$) and 
$$f: \Omega\times \mathbb{R}\times  \mathbb{R}^{n}\rightarrow
\mathbb{R}$$ 
is a Carath\'{e}odory function satisfying the following \emph{structural assumptions}
\begin{itemize}
\item[$(\mathbf{H1})$] $f: \Omega\times \mathbb{R}\times  \mathbb{R}^n\rightarrow
\mathbb{R}$ is convex in the pair $(s,\xi)$ 
 
 \vskip.3cm
\item[$(\mathbf{H2})$] for a.e.\,$x\in\Omega$ and for every $u\in \R$ and $\xi\in\R^n$ we have
\begin{equation}  \label{growth}
\sum_{i=1}^n \lambda_i (x)|\xi_i|^{p_i}\le {f}(x,u,\xi)\le \mu(x)\left\{|\xi|^{q}+|u|^{\gamma}+1\right\},
\end{equation}
where $1\leq p_1,\ldots,p_n\leq q\leq \gamma$, and 
$\lambda_1,\ldots,\lambda_n,\mu:\Omega\to(0,+\infty)$
are non-negative, measurable functions such that
\begin{equation} \label{eq:integrabilitylambdaimu}
\lambda_i ^{-1}\in L_{\mathrm{loc}}^{r_i}(\Omega )
 \,\,(\text{for $1\leq i\leq n$}),\qquad \mu \in L_{\mathrm{loc}%
}^{s}(\Omega ),
\end{equation}
for some exponents $r_i\in [1,\infty]$ and $s\in (1,\infty]$.
\end{itemize}
\vskip.3cm
Recall the definition of quasi-minimizers of functional in \eqref{functional}.
\begin{definition} \label{def:quasiminimizer}
A function $u\in W_{\rm loc}^{1,1}(\Omega)$ is a \emph{quasi-minimizer} of
 \eqref{functional} if
 the\-re exists some constant $Q \geq 1$ (depending on $u$) such that 
 \begin{itemize}
 \item[i)] $f(x,u,Du)\in L^1_{\rm loc}(\Omega)$;
 \vspace*{0.1cm}
 
 \item[ii)] for every $\varphi\in W^{1,1}(\Omega)$ with $\text{supp\ } \varphi \subset\Omega$ we have
\begin{equation*}
F(u;\text{supp\ }\varphi)\le  Q\,F(u+\varphi;\text{supp\ }\varphi),
\end{equation*}
\end{itemize}
\end{definition}

\bigskip
Let us now state  our main regularity  result.

\begin{theorem} \label{main-scalare} 
Let $\Omega\subseteq\R^n$ be a bounded open set, and let $f$
be
a
Ca\-ra\-th\'e\-o\-dory function satisfying assumptions  $(\mathbf{H1})$ and $(\mathbf{H2})$  
\emph{(}of which we inherit the notation\emph{)}. Setting
$\boldsymbol{p} = (p_1,\ldots,p_n)$ and
$$\Bdsigma := (\sigma_1,\ldots,\sigma_n),\qquad\text{with $\sigma_i := p_i\cdot\frac{r_i}{r_i+1}$},$$
(with the convention that $\frac{r_i}{r_i+1} = 1$ if $r_i = \infty$),  we assume that
\begin{equation} \label{eq:assumptionsParameters}
\begin{split}
\mathrm{(i)}\,\,\overline{\boldsymbol{\sigma}} < n\qquad
\mathrm{(ii)}\,\,q < \frac{\overline{\boldsymbol{\sigma}}^*}{s'} \qquad
\mathrm{(iii)}\,\,q\leq\gamma< 
 \frac{\overline{\Bdsigma}^*}{s'}\cdot	\frac{\overline{\boldsymbol{p}}}{q}+q-\overline{\boldsymbol{p}}.
\end{split}
\end{equation}
Then, any quasi-minimizer $u\in W^{1,1}_{\mathrm{loc}}(\Omega)$ of \eqref{functional} is 
locally bounded in $\Omega$.
\vspace*{0.05cm}
 
  More pre\-cisely, given any $x_0\in \Omega$ and any $R_0 < 1$ such that $B_{R_0}(x_0)\Subset\Omega$, 
  for every fixed $0< R \leq R_0$ we have the following estimate
\begin{equation}\label{stima-teorema}
\|u\|_{L^\infty(B_{R/2}(x_0))}
   \leq \frac{c}{R^{\vartheta_2}}
     \big[1+
    \|u\|_{L^{\overline{\Bdsigma}^*}(B_{R}(x_0))}\big]
    ^{\vartheta_1}
\end{equation}
where $\vartheta_1,\vartheta_2$ are explicitly given by
\[\vartheta_1:=\frac{\overline{\boldsymbol{\sigma}}^*\gamma-qs'\overline{\boldsymbol{p}}}{\overline{\boldsymbol{p}}\cdot\overline{\boldsymbol{\sigma}}^*-qs'(\gamma-q+\overline{\boldsymbol{p}})}\]
\[\vartheta_2:=\frac{q\overline{\boldsymbol{\Bdsigma}}^*}
{\overline{\boldsymbol{p}}\cdot\overline{\boldsymbol{\Bdsigma}}^*-{qs'}({\gamma-q}+
	\overline{\boldsymbol{p}})}\]
and 
 $c > 0$ is a constant only depending on the data in assumption $\mathbf{(H2)}$.
\end{theorem}
\begin{remark} \label{rem:SulThmPrincipale}
 A few remarks on Theorem \ref{main-scalare} are in order.
 \begin{enumerate}
   \item The second inequality in assumption \eqref{eq:assumptionsParameters} ensures that the range of \emph{admissible values} for $\gamma$
   in \eqref{eq:assumptionsParameters}-(iii) is non-empty: indeed, we have
   $$\frac{\overline{\Bdsigma}^*}{s'}\cdot	\frac{\overline{\boldsymbol{p}}}{q}+q-\overline{\boldsymbol{p}} > \overline{\boldsymbol{p}}+(q-\overline{\boldsymbol{p}}) = q.$$
 	\item In the particular case when Assumption $\mathbf{(H2)}$ holds with 
   \emph{$p_i=p=q$},
  then Assumption \eqref{eq:assumptionsParameters}-(ii) 
   is equivalent to \[\frac{1}{r}+\frac{1}{s}<\frac{p}{n},\]
   which is consistent with condition (1.9)  in \cite{BiCuMa}. 
   If, moreover,  we also have 
   $r_i=s=\infty$,   
Assumption \eqref{eq:assumptionsParameters}-(iii) boils down to 
$$p<n\quad\text{and}\quad \gamma<p^*$$
and  we have
\begin{align*}
(\ast)\quad & \vartheta_1= \frac{p^*\gamma-p^2}{p(p^*-\gamma)} \\ (\ast)\quad & \vartheta_2 = \frac{p^*}{p^*-\gamma}.
\end{align*}
\item The boundedness result in Theorem \ref{main-scalare} still holds if $f$ satisfies assumption $\mathbf{(H2)}$ \emph{with $0\leq \gamma < q $}. Indeed, in this case we have
  \begin{align*}
\mu(x)\left\{|\xi|^{q}+|u|^{\gamma}+1\right\}
  & \leq \mu(x)\left\{|\xi|^{q}+(1+|u|^{q})+1\right\}  	 \\
  & \leq 2\mu(x)\left\{|\xi|^{q}+|u|^{q}+1\right\},
  \end{align*}
  and this shows that $f$ satisfies $\mathbf{(H2)}$ \emph{also with the choice $\gamma = q$}. As a consequence, we can apply Theorem \ref{main-scalare} with $\gamma = q$, thus obtaining the following \emph{specialized}
  version of estimate \eqref{stima-teorema}:
  $$\|u\|_{L^\infty(B_{R/2}(x_0))}
   \leq \frac{c}{R^{\frac{q\overline{\boldsymbol{\Bdsigma}}^*}
     {\overline{\boldsymbol{p}}(\overline{\boldsymbol{\Bdsigma}}^*-qs')}}}
     \big[1+
    \|u\|_{L^{\overline{\Bdsigma}^*}(B_{R}(x_0))}\big]
    ^{\frac{q(\overline{\boldsymbol{\sigma}}^*-s'\overline{\boldsymbol{p}})}{\overline{\boldsymbol{p}}(\overline{\boldsymbol{\sigma}}^*-qs')}}
$$
(which holds true for every quasi-minimizer $u$ of \eqref{functional}, every $x_0\in\Omega$ and every $0<R\leq R_0\leq 1$ with $B_{R_0}(x_0)\Subset\Omega$).
 \end{enumerate}
\end{remark}

\section{Preliminary results} \label{Preliminar-results-scalar}
In order to keep the paper as self-contained as possible,
we collect here below all the preliminary (and mostly well-known) results which will be used
in the sequel.
Throughout what follows, we tacitly inherit all the notation introduced in the previous
sections (in particular, those appearing in the structural 
assumptions $\mathbf{(H1)}$\,-\,$\mathbf{(H2)}$).
\vspace*{0.1cm}

\noindent\textbf{1)\,\,Anisotropic Sobolev spaces.}
 Let $\mathcal{O}\subseteq\R^n$ be an arbitrary open set, and let 
 $\Bdbeta = (\beta_1,\ldots,\beta_n)$, with $1\leq \beta_1,\ldots,\beta_n\leq \infty$.
  The \emph{anisotropic Sobolev space}
  $W^{1,\Bdbeta}(\mathcal{O})$ is the space
  of functions defined as follows
\begin{equation*}
\begin{split}
& W^{1,\Bdbeta}(\mathcal{O})=\big\{ u\in W^{1,1}(\mathcal{O}):\,
\text{$u_{x_{i}}\in L^{\beta_{i}}(\mathcal{O})$ for all $1\leq i\leq n$}\big\},
\end{split}
\end{equation*}
 endowed with the norm
\begin{equation*}
\Vert u\Vert_{W^{1,\Bdbeta}}(\mathcal{O}):=
\Vert u\Vert_{L^{1}(\mathcal{O})}+\sum_{i=1}^{n}\Vert u_{x_{i}}\Vert_{L^{\beta_{i}}(\mathcal{O})}.
\end{equation*}
We also define
$$W_{0}^{1,\Bdbeta}(\mathcal{O})=
W_{0}^{1,1}(\mathcal{O})\cap W^{1,\Bdbeta}(\mathcal{O}).$$
Finally, we say that $u\in W^{1,\Bdbeta}_{\mathrm{loc}}(\mathcal{O})$ if
$$\text{$u\in W^{1,\Bdbeta}(\mathcal{V})$
for every open set $\mathcal{V}\Subset\mathcal{O}$}.$$
In view of the growth condition in assumption $\mathbf{(H2)}$, the anisotropic Sobolev spaces
are naturally related with the functional $F$ defined in \eqref{functional}; more precisely,
we have the following proposition.
\begin{proposition}\label{crescitaPaolo}
Let $\Omega\subseteq\R^n$ be a bounded open set, and let $f$
be a Ca\-ra\-th\'e\-odory function satisfying assumption $\mathbf{(H2)}$. We set
$$\boldsymbol{\sigma} = (\sigma_1,\ldots,\sigma_n),\quad
\text{with $\sigma_i = p_i \cdot\frac{r_i}{r_i+1}\,\,(1\leq i\leq n)$}$$
\emph{(}here, $p_i$ and $r_i$ are the exponents introduced 
in \eqref{eq:integrabilitylambdaimu}\emph{)}.

Then the following estimate holds
\begin{align}
& \frac{1}{n}\sum_{i=1}^{n} \Vert \lambda_i^{-1}\Vert_{L^{r_{i}}(\Omega ^{\prime })}^{-1}\left\Vert
v_{x_{i}}\right\Vert^{p_{i}}_{L^{\sigma_i}(\Omega^{\prime})} 
\leq \int_{\Omega^{\prime }}f(x,v,Dv)\,dx \label{stima-con-funzionale-da-sotto} 
\end{align} 
for every $v\in W_{\rm loc}^{1,1}(\Omega)$  
and every bounded open set $\Omega'\Subset\Omega$.
\vspace*{0.1cm}

As a consequence, if $u\in W^{1,1}_{\mathrm{loc}}(\Omega)$ is a quasi\,-\,minimizer of the functional
$F$ associated with $f$ and defined in \eqref{functional}, we have
\begin{equation} \label{eq:WFinWsgima}
 u\in W^{1,\boldsymbol{\sigma}}_{\mathrm{loc}}(\Omega).
\end{equation}
\end{proposition}
\begin{proof} 
Let  $\Omega',v$ be as in the statement. We distinguish two cases.
\medskip

\noindent \textsc{Case I)}: $r_i<\infty$. In this case, by exploiting
the classical H\"older inequality with exponents $\frac{r_i+1}{r_i}$ and $\displaystyle r_i+1$, 
we have
\begin{equation} \label{stima-funzionale1}
\begin{split}
\int_{\Omega'}|v_{x_i}|^{\frac{p_ir_i}{r_i+1}}\,dx=&\int_{\Omega'}\lambda_i^{-\frac{r_i}{r_i+1}}\lambda_i^{\frac{r_i}{r_i+1}}|v_{x_i}|^{\frac{p_ir_i}{r_i+1}}\,dx
\\ 
\le & \Vert \lambda_i^{-1}\Vert_{L^{r_{i}}(\Omega ^{\prime })}^{\frac{r_i}{r_i+1}}
\left(\int_{\Omega'}\lambda_i |v_{x_i}|^{p_{i}}\,dx\right)^{\frac{r_i}{r_i+1}}.	
\end{split}	
\end{equation}
By the growth condition \eqref{growth}, we have
$$
\int _{\Omega'} \lambda_i |v_{x_i}|^{p_{i}}\,dx  \leq  \int_{\Omega'}f(x,v, Dv)\,dx;
$$
as a consequence, we obtain
	\begin{equation} \label{stima-funzionale1bis}
	\left(\int_{\Omega'} |v_{x_i}|^{\frac{p_i r_i}{r_i+i}}\,dx \right)^{\frac{r_i+1}{r_i }} \leq  \Vert \lambda_i^{-1}\Vert_{L^{r_{i}(\Omega ^{\prime })}}\int_{\Omega'}f(x,v, Dv)\,dx,
\end{equation}

\noindent \textsc{Case II):} $r_i = \infty$. In this case, the convention $\frac{r_i}{r_i+1}=1$ applies. Thus,
	\begin{equation*}
		\begin{split}
		\int_{\Omega'}|v_{x_i}|^{\frac{p_ir_i}{r_i+1}}\,dx=&
				\int_{\Omega'}|v_{x_i}|^{p_i}\,dx=	
		\int_{\Omega'}\lambda_i^{-1}\lambda_i |v_{x_i}|^{p_i}\,dx
		\\ 
		\le & \Vert \lambda_i^{-1}\Vert_{L^{\infty}(\Omega ^{\prime })}		\int_{\Omega'}\lambda_i |v_{x_i}|^{p_{i}}\,dx.	
		\end{split}
	\end{equation*}
	that is \eqref{stima-funzionale1} holds true also for $r_i=\infty$. Therefore, reasoning as above, we conclude that \eqref{stima-funzionale1bis} holds also in this case. 
\medskip

From \eqref{stima-funzionale1bis}, adding with respect to $i$ we readily derive the desired \eqref{stima-con-funzionale-da-sotto}.

\medskip

Let now $u\in W^{1,1}_{\mathrm{loc}}(\Omega)$ be a quasi-minimizer of
 \eqref{functional} (in the sense of Definition \ref{def:quasiminimizer}).
 Since, by definition we know that
 $$f(x,u,Du)\in L^1_{\mathrm{loc}}(\Omega),$$
 from \eqref{stima-con-funzionale-da-sotto} we immediately derive \eqref{eq:WFinWsgima}.
\end{proof}

As for \emph{classical case $\beta_1 = \ldots = \beta_n = \beta$}
considered in \cite{BiCuMa} (and in view of Pro\-position
\ref{crescitaPaolo}),
the following Em\-bedding Theorem for 
$W^{1,(\beta_1,\ldots,\beta_n)}(\mathcal{O})$ 
will play key role in our argument;
for a proof we refer, e.g., to \cite{BLPV,tro}.
\begin{theorem}\label{immersione-anisotropo}
 Let
 $\mathcal{O}\subseteq\R^n$ 
 be a \emph{bounded} open set, and let
 $\Bdbeta = (\beta_1,\ldots,\beta_n)$, with
 $1\leq \beta_1\leq\ldots\leq\beta_n\leq\infty$. We suppose that
 $$\overline{\boldsymbol{\beta}} < n.$$
 Then, the following facts hold.
 \begin{enumerate}
	\item[1)]  There exists a constant $c > 0$, only depending on 
 $n, \Bdbeta$, such that 
 \begin{equation*}
\|u\|_{L^{{\overline{\Bdbeta}^*}}(\mathcal{O} )}\leq c  
\left(\Pi_{i=1}^{n} \|u_{x_i}\|_{L^{\beta_i}(\mathcal{O} )}
\right)^{\frac{1}{n}}\quad\text{}
\end{equation*}
for every $u\in W_0^{1,\Bdbeta}(\mathcal{O})$.

\item[2)] Assume that $\max_i\beta_i < \overline{\Bdbeta}^*$. Then, for every open set $\mathcal{V}\Subset\mathcal{O}$
we have the following \emph{continuous embedding} of $W^{1,\Bdbeta}(\mathcal{O})$:
$$W^{1,\Bdbeta}(\mathcal{O})\hookrightarrow L^{\overline{\Bdbeta}^*}(\mathcal{V}).$$
 \end{enumerate}
\end{theorem}
By combining the previous Proposition \ref{crescitaPaolo} with 
the ``anisotropic'' embed\-ding in Theo\-rem \ref{immersione-anisotropo}, we derive
the following simple
yet crucial ``higher-inte\-grability'' result
for quasi-minimizers of \eqref{functional}.
\begin{proposition} \label{prop:HigherIntegrab}
Let the assumptions and the notation of Theorem \ref{main-scalare} be in force,
and let $u\in W^{1,1}_{\mathrm{loc}}(\Omega)$
be a quasi-minimizer of \eqref{functional}. Then,
$$u\in L^{\overline{\boldsymbol{\sigma}}^*}_{\mathrm{loc}}(\Omega).$$
Hence, in particular, $u\in L^{qs'}_{\mathrm{loc}}(\Omega)$.
\end{proposition}
\begin{proof}
 First of all we observe that, since $u$ is a quasi-minimizer of \eqref{functional}
 and since $f$ satisfies assumption $\mathbf{(H2)}$, we know from
 Proposition \ref{crescitaPaolo} that
 $$u\in W^{1,\boldsymbol{\sigma}}_{\mathrm{loc}}(\Omega).$$
 On the other hand, owing to \eqref{eq:assumptionsParameters}-(ii), we have
 \begin{equation}\label{eq:sigmaisigmastar}
  \sigma_i = p_i\cdot\frac{r_i}{r_i+1}\leq p_i\leq q\leq qs' < \overline{\boldsymbol{\sigma}}^*\quad
 \text{for al $1\leq i\leq n$}.
 \end{equation}
 In view of this fact, 
 we are then entitled to apply the anisotropic Embedding Theorem in Theorem \ref{immersione-anisotropo}-2),
 ensuring that
 $u\in L^{\overline{\boldsymbol{\sigma}}^*}_{\mathrm{loc}}(\Omega)$.
 From this, since 
 $$qs' < \overline{\boldsymbol{\sigma}}^*,$$ 
  we also derive that $u\in L^{qs'}_{\mathrm{loc}}(\Omega)$, as desired.
\end{proof}

The next proposition is a Poincar\'e-Sobolev\,-\,type inequality for the aniso\-tro\-pic setting. 
From now on, we adopt the notation already introduced in Pro\-position
\ref{crescitaPaolo}: with reference to assumption $\mathbf{(H2)}$, we set
$$\boldsymbol{\sigma} = (\sigma_1,\ldots,\sigma_n),\quad
\text{with $\sigma_i=p_i \cdot \frac{r_i}{r_i+i}\,\,(1\leq i\leq n)$};$$ 
accordingly, $\overline \Bdsigma$ is the harmonic mean of the $\sigma_i$'s as defined  
in \eqref{definizione-media}, and $\overline{\Bdsigma}^{\ast}$ is the Sobolev exponent of  $\overline{\Bdsigma}$
as defined in \eqref{p*}.
\begin{proposition}\label{l:step1}
Let $\Omega\subseteq\R^n$ be a bounded open set, and let $f$ be a Ca\-ra\-th\'e\-odory function
satisfying assumption $\mathbf{(H2)}$. Moreover, let $\Omega'$ be a bounded open set such 
that $\Omega'\Subset\Omega$. 

Then,
there exists a positive constant $c > 0$ such that 
\begin{equation*}
\begin{split}
& \left(\int_{\Omega'} |v|^{{\overline \Bdsigma}^{\ast}}\,dx\right)^{\frac{1}{{\overline{\Bdsigma}^{\ast }}}} \\
& \qquad \leq c \left\lbrace \prod_{i=1}^{n}\left[ \|\lambda_i^{-1}\|_{L^{r_{i}}(\Omega')}
  \int_{\Omega'} \lambda_i\,
|v_{x_{i}}|^{p_{i}}\,dx \right]^{\frac{1}{p_{i}}}\right\rbrace ^{\frac{1}{n}},
\end{split}
\end{equation*}
for every $v\in W_0^{1,\boldsymbol{\sigma}}(\Omega')$.
\end{proposition} 
\vskip.3cm
\begin{proof}
The inequality follows by collecting Theorem \ref{immersione-anisotropo} and  \eqref{stima-funzionale1}.
\end{proof}

\noindent \textbf{2)\,\,Some algebraic inequalities.} In the next lemma we collect
some use\-ful algebraic inequalities among the parameters involved in Theorem \ref{main-scalare},
which are consequences of assumption \eqref{eq:assumptionsParameters}.
 \begin{lemma}\label{integrabilit-u}
Let the assumptions and the notation of Theorem \ref{main-scalare} be in force.
Then, the following assertions hold.
\begin{enumerate} 
 \item $1\leq \overline{\boldsymbol{p}}\leq q$;
 \item for every $1\leq i\leq n$, we have $\lambda_i\leq 2\mu$ a.e.\,in $\Omega$;
 \end{enumerate}
\end{lemma}
\begin{proof}[Proof of \emph{(1)}] Since, by assumption, $1\leq p_1,\ldots,p_n\leq q$, we get
$$1 = \frac{1}{n}\cdot n \geq \frac{1}{n}\sum_{i = 1}^n\frac{1}{p_i}
\geq \frac{1}{n}\cdot\frac{n}{q} = \frac{1}{q};$$
thus, by definition of
$\overline{\boldsymbol{p}}$ we immediately get $1\leq \overline{\boldsymbol{p}}\leq q$.
\medskip

\noindent \emph{Proof of} (2). By choosing $\xi = e_i$ (where $e_i$ denotes the $i$-th vector
of the canonical basis of $\R^n$) and $u = 0$ in \eqref{growth}, we immediately obtain
\begin{align*}
 \lambda_i(x) &=\lambda_i(x)|\xi_i|^{p_i}  = \sum_{j = 1}^n\lambda_j(x)|\xi_j|^{p_i}
 \\
 &\leq f(x,0,\xi) \leq \mu(x)\{|\xi|^q+1\} = 2\mu(x).
\end{align*}
This ends the proof.
\end{proof}

\noindent\textbf{3)\,\,Real-analysis lemmas.}
To prove the regularity result in Theorems \ref{main-scalare} we will use the following two 
classical lemmas (see, e.g., \cite{giusti} for a proof).
\begin{lemma}\label{lemma-giusti}
Let $\phi:[\tau_0,\tau_1]\to\R$ be  a non-negative and bounded function 
satisfying, for every $\tau_0\leq s < t\leq \tau_1$, the following estimate
$$
\phi(s) \leq\theta \phi(t)+ \frac{A}{(t-s)^\alpha}+ B,
$$
where $A,B,\alpha$ are non-negative constants and $0<\theta<1$.  

Then, for all $\rho$ and $R$, such that $\tau_0\leq \rho \leq R \leq \tau_1$, we have
$$
\phi(\rho) \leq C \left\{ \frac{A}{(R-\rho)^\alpha} +B \right\}.
$$
\end{lemma}

\begin{lemma}\label{lemma-giusti2}
Let $\alpha >0$, and let $(J_h)$ be a sequence of positive real numbers
satisfying, for some $A > 0$ and $\lambda > 1$, the following estimate
$$
J_{h+1} \leq A\,\lambda^h J_h^{1+\alpha}\quad\text{for all $h\geq 0$}.
$$
If  $J_0 \leq A^{-\frac{1}{\alpha}}\lambda^{-\frac{1}{\alpha^2}}$, 
then $$
J_h \le \lambda^{-\frac{h}{\alpha}}J_0.
$$
Hence, in particular, we have
$$\lim_{h\to \infty} J_h=0.$$
\end{lemma}
\section{A Caccioppoli\,-\,type inequality for \eqref{functional}} \label{s:Caccioppoli1}
In this section we establish a Caccioppoli\,-\,type inequality
for the quasi-mi\-ni\-mizers of the functional ${F}$ in \eqref{functional}; as it is natural
to expect, this result will be a key tool for proving Theorem \ref{main-scalare}.

In order to properly state and prove the aforementioned Caccioppoli\,-\,type inequality,
we first fix a notation: given any function 
$u \in W^{1,1}_{\rm loc}(\Omega)$  and any Euclidean
ball $B_R(x_0) \subseteq \Omega$, we define the super-level sets:
$$
A_{k,R}:=\{x \in B_R(x_0)\,:\, u(x)>k \},\qquad k\in \mathbb{R}.$$
\begin{proposition}\label{degiorgi} 
Let $f:\Omega\times\R\times\R^n\to\R$ be a Carath\'eodory function
sati\-sfy\-ing $\mathbf{(H1)}$ and the following
\emph{one\,-\,side version} of $\mathbf{(H2)}$, that is,
\begin{equation}  \label{growth-caccioppoli}
0\le {f}(x,u,\xi)\le \mu(x) \left\{|\xi|^{q}+|u|^{\gamma}+1\right\},
\end{equation}
with $1\leq q\leq\gamma$ and $\mu \in L_{\mathrm{loc}}^s
(\Omega)$ \emph{(}for some $s\in(1,\infty]$\emph{)}.
Moreover, let 
\begin{equation} \label{eq:IntegrabilityAss}
 u \in W_{\rm loc}^{1,1}(\Omega)\cap L_{\rm loc}^{qs'}(\Omega)
 \end{equation}
be a quasi-minimizer of  $F$, and let $B = B_R(x_0)\subset\Omega$.

Then, for every $0<\rho<R $ and every $k\ge 1$ we have
\begin{equation} \label{degiorgi1}
\begin{split}
& \int_{A_{k,\rho}} f(x,u,Du) \,dx \\
& \qquad 
\leq C\Big\{\frac{1}{(R-\rho)^q}\int_{A_{k,R}} \mu(x)\big(
(u-k)^q+k^\gamma\big)\,dx
\\
& \qquad\qquad\qquad
 + \|\mu\|_{L^s(B_R(x_0))}|A_{k,R}|^{1/s'}\Big\},
\end{split}
\end{equation}
 where $c > 0$ is a constant only depending on $n,q,Q$.  
\end{proposition}
\begin{proof}  
Let $\rho,k$ be as in the statement of the proposition, and let $s,t > 0$ be such that
 $\rho\leq s<t\leq R$. Moreover, let
$\eta\in C_0^{\infty}(B_t)$ be a
cut-off function satisfying the following assumptions:
\begin{equation}  \label{eta}
0\le \eta\le 1,\quad \eta \equiv 1\ \text{in $B_{s}(x_0)$,} \quad \text{$|D \eta|\le \frac{2}{t-s}$.}
\end{equation}
Accordingly, we define  
$$
w:=\max(u-k,0) \quad \textrm{and}\quad \varphi=-\eta^{q}w.
$$ 
We explicitly observe that, by definition, we have
\begin{itemize}
 \item[a)] $\varphi\in W^{1,1}(\Omega)$;
 \item[b)] $\varphi = -(u-k) = k-u < 0$ on $A_{k,s} = B_s(x_0)\cap\{u > k\}$;
 \item[c)] $\mathrm{supp}(\varphi) \subseteq B_{t}(x_0)\cap\{u > k\} = A_{k,t}\subset \Omega$.
\end{itemize}
In view of these facts, 
by Definition \ref{def:quasiminimizer} we get 
\begin{equation} \label{e:postconv} 
\begin{split}
\int_{A_{k,{s}}} f(x,u,Du) \,dx
 & \leq \int_{\mathrm{supp}(\varphi)} f(x,u,Du)\,dx  \\
 & = F(u;\mathrm{supp}(\varphi)) 
  \leq Q\,F(u+\varphi;\mathrm{supp}(\varphi))\\
 &  = 
Q\int_{A_{k,{t}}} f(x,u+\varphi, Du+D\varphi)\,dx = (\ast).
\end{split}
\end{equation}
We now observe that, for a.e.\,$x\in A_{k,{t}}\cap \{\eta=0\}$, we have
\[f(x,u+\varphi, Du+D\varphi)=f(x,u, Du)  = (1-\eta^q)f(x,u,Du).\]
Moreover, by the convexity of $f$, for a.e.\,$x\in A_{k,{t}}\cap \{\eta>0\},$ we have 
\begin{align*}
& f(x,u+\varphi,Du+D\varphi) \\
& \qquad = f\left(x,(1-\eta^{ q})u+ \eta^{ q}k,(1-\eta^{ q})Du + q \eta^{ q-1}(k-u)D\eta\right)
\\ 
& \qquad \le (1-\eta^{ q})f\left(x, u ,Du \right)
+ \eta^{ q} f\left(x,k, q \eta^{-1}(k-u)D\eta\right);
\end{align*}
 thus, by the growth assumption \eqref{growth-caccioppoli}, we obtain
\begin{equation}  \label{fuori2}
\begin{split}
(\ast) & \leq  Q\int_{A_{k,{t}}}(1-\eta^{ q})f(x,u,Du)\,dx 
\\ & \qquad +
\int _{A_{k,{t}}}  
 \mu(x)\left\{\frac{2^q}{(t-s)^q} 
(u-k)^q+ k^\gamma+1\right\}\,dx.
\end{split}
\end{equation}
As a consequence, 
by combining \eqref{e:postconv}\,-\,\eqref{fuori2},
and by recalling the properties of $\eta$ in \eqref{eta},
   we derive the following estimate
 \begin{equation} \label{eq:toconcludeCaccioppoli}
\begin{split}
&\int_{A_{k,s}} f(x,u,Du) \,dx  \le 
Q\int_{A_{k,t}\setminus A_{k,s}} f(x,u,Du)\,dx
\\
&\qquad +\frac{c}{(t-s)^{q}}\int_{A_{k,R}} \mu(x)\big(
(u-k)^q+k^\gamma\big)\,dx
\\
& \qquad\qquad+ \|\mu\|_{L^s(B_R(x_0))}|A_{k,R}|^{1/s'},
\end{split}
\end{equation}
where $c > 0$ is a constant only depending on $q$ and $Q$.

With the above estimate \eqref{eq:toconcludeCaccioppoli} at hand,
we can easily complete the proof
of the proposition. Indeed, starting from the cited \eqref{eq:toconcludeCaccioppoli}
and adding to both sides of this estimate the term 
$$Q\,\int_{A_{k,s}} f(x,u,Du) \,dx$$
(notice that this integral is finite, since $u$ is a quasi-minimizer of \eqref{functional}
and thus $f(x,u,Du)$ is integrable on $A_{k,s}\subseteq B_s(x_0)\Subset\Omega$), we clearly get
\begin{align*}
&\int_{A_{k,s}} f(x,u,Du) \,dx \leq 
\frac{Q}{Q+1} \int_{A_{k,t}} f(x,u,Du) \,dx
\\
&\qquad +\frac{c}{(t-s)^{q}}\int_{A_{k,R}} \mu(x)\big(
(u-k)^q+k^\gamma\big)\,dx
+ \|\mu\|_{L^s(B_R(x_0))}|A_{k,R}|^{1/s'};
\end{align*}
then, by applying  Lemma \ref{lemma-giusti} with the choice 
\begin{align*} 
\mathrm{i)}\,\,&\text{$\tau_0:= \rho$ and $\tau_1:=R$}; \\
\mathrm{ii)}\,\,&\phi(r) =\int_{{A_{k,r}}} f(x,u,Du) \,dx; \\
\mathrm{iii)}\,\,&\theta = \frac{Q}{Q+1}\in(0,1); \\
\mathrm{iv)}\,\,&\alpha = q\,\,\text{and}\,\,A = c\,\int_{A_{k,R}} \mu(x)\big(
(u-k)^q+k^\gamma\big)\,dx; \\
\mathrm{v)}\,\,& B = \|\mu\|_{L^s(B_R(x_0))}|A_{k,R}|^{1/s'};
\end{align*}
we obtain the desired \eqref{degiorgi1}, and the proof is complete.
\end{proof}
\begin{remark} \label{rem:finitezza}
 We explicitly notice that the integrability assumption 
 \eqref{eq:IntegrabilityAss} ensures
 that the right-hand side of \eqref{degiorgi1} is finite. Indeed, since
 $\mu\in L^s_{\mathrm{loc}}(\Omega)$, by
 the classical H\"older inequality we get
 \begin{align*}
  & \int_{A_{k,R}} \mu(x)\big(
(u-k)^q+k^\gamma\big)\,dx 
\leq
\int_{A_{k,R}} \mu(x)(
|u|^q+k^\gamma)\,dx \\
& \qquad \leq
\int_{B_R(x_0)}\mu(x)(
|u|^q+k^\gamma)\,dx \\
& \qquad (\text{by H\"older's inequality, since $\mu\in L^s(B_R(x_0))$}) \\
&\qquad\leq
\|\mu\|_{L^s(B_R(x_0))}\big(\|u\|^q_{L^{qs'}(B_R(x_0))}
+ |B_R(x_0)|^{1/s'}k^\gamma\big).
 \end{align*}
 On the other hand, when $f$ satisfies the assumption $\mathbf{(H2)}$ in its full strength
 (as required in the statement of Theorem \ref{main-scalare}), we know
 from Proposition \ref{prop:HigherIntegrab} that any quasi-minimizer
 of \eqref{functional} satisfies \eqref{eq:IntegrabilityAss}.
\end{remark}
\section{Proof of Theorem \ref{main-scalare}}
\label{s:proofs}
Thanks to all the results established so far,
we are finally ready to provide the proof
of Theorem \ref{main-scalare} (of which we inherit all the assumptions and the notation).
Throughout what follows, we simply denote by
$$c(\mathbf{data})$$
any constant (possibly different from line to line) which only depends on the 
\emph{data in assumption $\mathbf{(H2)}$}, that is, the exponents $p_i,q,\gamma$ and the $L^{r_i}$ and $L^{s}$-norms
of $\lambda_i^{-1},\mu$, respectively, on some fixed compact subset of $\Omega$. 
\begin{proof}[Proof (of Theorem \ref{main-scalare})]
Let 
$u\in W^{1,1}_{\mathrm{loc}}(\Omega)$ be a quasi-minimizer of
 \eqref{functional}, and let $x_0\in\Omega,\,R_0\leq 1$ be such that $B_0 = B_{R_0}(x_0)\Subset\Omega$.

Owing to Proposition \ref{prop:HigherIntegrab}, 
we know that 
\begin{equation} \label{eq:uqsprimeref}
 u\in L^{qs'}_{\mathrm{loc}}(\Omega);
 \end{equation}
hence, for every fixed $0<R\leq R_0$ we can define
\begin{equation*}
  J_h := \int_{A_{k_h,\rho_h}}(u-k_h)^{qs'}\,d x\qquad (h\geq 0),
 \end{equation*}
 where $(\rho_h)_h$ and $(k_h)_h$ are given, respectively, by
 \begin{align*}
 (\ast)\,\,&\rho_{h}:=\frac{R}{2}+\frac{R}{2^{h+1}}=\frac{R}{2}\Big(1+\frac{1}{2^h}\Big); \\
 (\ast)\,\,&  k_h:= d\left( 1-\frac{1}{2^{h+1}}\right) 
 \end{align*}
 (for some $d\geq 2$ to be chosen later on). Moreover, we also define
$$
\overline{\rho}_h:=\frac{\rho_{h}+\rho_{h+1}}{2}=\frac{
 R}{2}\Big(1+\frac{3}{4\cdot 2^h}\Big).
$$ 
 We explicitly observe that, since 
   $(\rho_h)_h$ is decreasing and $(k_h)_h$ is increasing,
   the sequence $(J_h)_h$ is decreasing: in fact, we have
   \begin{equation}
   \begin{split}
    J_{h+1} & = \int_{A_{k_{h+1},\rho_{h+1}}}(u-k_{h+1})^{qs'}\,d x
    \leq \int_{A_{k_{h+1},\rho_h}}(u-k_{h+1})^{qs'}\,d x \\[0.1cm]
    & \leq \int_{A_{k_{h+1},\rho_h}}(u-k_{h})^{qs'}\,d x 
    \leq \int_{A_{k_{h},\rho_h}}(u-k_{h})^{qs'}\,d x = J_h.
    \end{split}
    \label{e:Jh+1leJhGC}
   \end{equation}
   Furthermore, 
    since $\rho_h\leq R\leq R_0$ we have
   \begin{equation} \label{eq.Jhleqone}
   \begin{split}
    J_h & = \int_{A_{k_h,\rho_h}}(u-k_h)^{qs'}\,d x
    \leq \int_{A_{k_h,\rho_h}}|u|^{qs'}\,d x \\[0.1cm]
    & \leq \int_{B_{R}(x_0)}|u|^{qs'}\,d x = \|u\|_{L^{qs'}(B_R(x_0))}^{qs'}<+\infty. 
   \end{split}
   \end{equation}
All that being said, we now turn to prove the following claim,
which is the real core of our argument: \emph{there exists a
constant $\mathbf{C} >0$, depending on the data appearing in assumption $\mathbf{(H2)}$, such that}
\begin{equation} \label{stimaJh}
J_{h+1}\leq \mathbf{C}\big[1+
    \|u\|^{(\gamma-q)s'}_{L^{\overline{\Bdsigma}^*}(B_R(x_0))}\big]^{\frac{q}{\overline{\boldsymbol{p}}}}
    \cdot\frac{1}{d^{\delta_1}R^{\delta_2}}\cdot \lambda^h J_{h}^{1+\alpha},\quad\text{for all $h\geq 0$},
\end{equation}
for suitable constants $\delta_1,\delta_2,\alpha > 0$ and $\lambda> 1$.
\vspace*{0.15cm}
 
\noindent To begin with, we choose a sequence  $(\zeta_{h})_h$  of  cut-off functions
such that
\begin{equation}\label{prop-zeta}
\begin{split}
\mathrm{i)}\,\,&\zeta_h\in C_c^{\infty}(B_{\overline{\rho}_{h}}(x_0)),\\ 
\mathrm{ii)}\,\,&\text{$\zeta_h\equiv 1$ in $B_{\rho_{h+1}}(x_0)$},\\ 
\mathrm{iii)}\,\,&
|D\zeta_h| \leq 2^{h+4}/R.
\end{split}
\end{equation}
Moreover, to keep the notation as simple as possible, we set
$$u_h = (u-k_h)_+ = \max\{u-k_h,0\}.$$
Now, starting from the very
definition of $J_h$, and applying  the H\"older inequa\-lity 
with exponent $\overline{\boldsymbol{\sigma}}^*/(qs') > 1$
(see assumption \eqref{eq:assumptionsParameters}-(ii)), we get
\begin{equation*}
\begin{split}
J_{h+1}  &
\leq |A_{k_{h+1},{\rho}_{h+1}}|^{1-\frac{qs'}{\overline{\Bdsigma}^{\ast}}} 
\left( \int_{A_{{k_{h+1},\bar{\rho}_{h}}}}((u-k_{h+1})\zeta_h)^{\overline{\Bdsigma}^{\ast}} \,dx\right)^{\frac{q s'}{\overline{\Bdsigma}^{\ast}}} 
\\&
=
 |A_{k_{h+1},{\rho}_{h+1}}|^{1-\frac{qs'}{\overline{\Bdsigma}^{*}}} \left( \int_{B_{\bar{\rho}_{h}}(x_0)} 
 (u_{h+1}\zeta_h )^{\overline{\Bdsigma}^{*}} \,dx\right)^{\frac{qs'}{\overline{\Bdsigma}^{*}}},
\end{split}
\end{equation*}
where we have also used the properties of $\zeta_h$ in \eqref{prop-zeta}.
\vspace*{0.1cm}

On the other hand, since we know that $u\in W^{1,\Bdsigma}_{\mathrm{loc}}(\Omega)\cap 
L^{\overline{\Bdsigma}^*}_{\mathrm{loc}}(\Omega)$ (see Pro\-po\-si\-tions \ref{crescitaPaolo} and
\ref{prop:HigherIntegrab}, respectively), it is readily seen that
$$v := u_{h+1}\zeta_h\in W_0^{1,\Bdsigma}(B_{\bar{\rho}_{h}}(x_0))$$
(also recall that $\sigma_i\leq \overline{\Bdsigma}*$, see
\eqref{eq:sigmaisigmastar}); as a consequence, we are entitled to apply 
Proposition \ref{l:step1} (with $\Omega' = B_{\bar{\rho}_{h}}(x_0)\Subset\Omega$),
obtaining
\begin{equation} \label{eq:doveusarestimaJi}
\begin{split}
J_{h+1}&\leq c |A_{k_{h+1},{\rho}_{h+1}}|^{1-\frac{qs'}{\overline{\Bdsigma}^{\ast}}}\times \\
&\quad\times \Big\{\displaystyle
 \prod_{i=1}^{n}\Big[\|\lambda_i^{-1}\|_{L^{r_i}(B_{R_0}(x_0))}
  \int_{B_{\bar{\rho}_{h}}(x_0)} \lambda_i\,
|v_{x_{i}}|^{p_{i}}\,dx \Big]^{\frac{1}{p_{i}}}\Big\} ^{\frac{qs'}{n}} \\
& \leq c(\mathbf{data})|A_{k_{h+1},{\rho}_{h+1}}|^{1-\frac{qs'}{\overline{\Bdsigma}^{\ast}}}\times
\\
&\quad \times
\prod_{i = 1}^n\Big[\int_{B_{\bar{\rho}_{h}}(x_0)} \lambda_i(x)\,
|(u_{h+1}\zeta_h)_{x_{i}}|^{p_{i}}\,dx \Big]^{\frac{qs'}{p_{i}n}}.
\end{split}
\end{equation}
To proceed further, we set
$$\mathfrak{J}_i = \int_{B_{\bar{\rho}_{h}}(x_0)} \lambda_i(x)\,
|(u_{h+1}\zeta_h)_{x_{i}}|^{p_{i}}\,dx \qquad (1\leq i\leq n),$$
and we observe that, by combining the growth condition in assumption
$\mathbf{(H2)}$ with the Caccioppoli-type inequality in
Proposition \ref{degiorgi}, we have
\begin{align*}
 \mathfrak{J}_i & \leq
  c(\mathbf{data}) \int_{B_{\bar{\rho}_{h}}(x_0)} \lambda_i(x)\,
\big\{|(\zeta_h)_{x_i}|^{p_i}u_{h+1}^{p_i}+
\zeta_h^{p_i}|(u_{h+1})_{x_i}|^{p_i}\big\}\,dx
\\
& (\text{by properties of $\zeta_h$ in \eqref{prop-zeta}}) \\
& \leq c(\mathbf{data})\Big(\frac{2^{h+4}}{R}\Big)^{p_i}
\int_{A_{k_{h+1},\bar{\rho}_h}}\lambda_i(x)(u-k_{h+1})^{p_i}\,dx
\\
& \qquad+c(\mathbf{data})\int_{A_{k_{h+1},\bar{\rho}_h}} \lambda_i(x)|u_{x_i}|^{p_i}\,dx \\
& (\text{here we use assumption $\mathbf{(H2)}$}) \\
& \leq  c(\mathbf{data})\Big(\frac{2^{h+4}}{R}\Big)^{p_i}
\int_{A_{k_{h+1},\bar{\rho}_h}}\lambda_i(x)(u-k_{h+1})^{p_i}\,dx
\\
& \qquad+c(\mathbf{data})\int_{A_{k_{h+1},\bar{\rho}_h}} f(x,u,Du)\,dx \\
& (\text{by the Caccioppoli-type inequality in Proposition \ref{degiorgi}}) \\
& \leq c(\mathbf{data})\Big(\frac{2^{h+4}}{R}\Big)^{p_i}
\int_{A_{k_{h+1},\bar{\rho}_h}}\lambda_i(x)(u-k_{h+1})^{p_i}\,dx\\
& \qquad
+ c(\mathbf{data})\Big(\frac{2^{h+3}}{R}\Big)^q\int_{A_{k_{h+1},\rho_h}} \mu(x)\big(
(u-k_{h+1})^q+k_{h+1}^\gamma\big)\,dx
\\
& \qquad\qquad
 +  c(\mathbf{data})\|\mu\|_{L^s(B_R(x_0))}|A_{k_{h+1},\rho_h}|^{1/s'}.
\end{align*}
From this, since $\lambda_i\leq 2\mu$ a.e.\,in $\Omega$ (see 
Lemma \ref{integrabilit-u}), and taking into account that
$p_1,\ldots,p_n\leq q$ (by assumption $\mathbf{(H2)}$), 
we obtain
\begin{align*}
 \mathfrak{J}_i & 
 \leq c(\mathbf{data})\Big(\frac{2^{h+4}}{R}\Big)^q
 \int_{A_{k_{h+1},{\rho}_h}}\mu(x)\big[1+(u-k_{h+1})^{q}\big]\,dx\\
 & \qquad
+ c(\mathbf{data})\Big(\frac{2^{h+3}}{R}\Big)^q\int_{A_{k_{h+1},\rho_h}} \mu(x)\big(
(u-k_{h+1})^q+k_{h+1}^\gamma\big)\,dx
\\
& \qquad\qquad
 +  c(\mathbf{data})\|\mu\|_{L^s(B_R(x_0))}|A_{k_{h+1},\rho_h}|^{1/s'}\\
 & \leq c(\mathbf{data})\Big(\frac{2^h}{R}\Big)^q\times \\
 &\qquad\times\Big\{\int_{A_{k_{h+1},\rho_h}} \mu(x)\big(
(u-k_{h+1})^q+k_{h+1}^\gamma\big)\,dx+|A_{k_{h+1},\rho_h}|^{1/s'}\Big\},
\end{align*}
where we have also used the fact that $\bar{\rho}_h\leq\rho_h$ and $k_{h+1}\geq d/2\geq 1$.
\vspace*{0.1cm}

Summing up, by combining this last estimate with \eqref{eq:doveusarestimaJi},
and by exploi\-ting H\"o\-l\-der's inequality with exponent $s$
(see \eqref{eq:uqsprimeref}), we conclude that
\begin{equation} \label{eq:daquicomeBCM}
 \begin{split}
  J_{h+1}& \leq c(\mathbf{data})|A_{k_{h+1},{\rho}_{h+1}}|^{1-\frac{qs'}{\overline{\Bdsigma}^{\ast}}}
  \cdot \prod_{i = 1}^n\mathfrak{J}_i^{\frac{qs'}{np_i}}\\
& \leq c(\mathbf{data})|A_{k_{h+1},{\rho}_{h+1}}|^{1-\frac{qs'}{\overline{\Bdsigma}^{\ast}}}
\Big(\frac{2^h}{R}\Big)^{\frac{q^2s'}{\overline{\boldsymbol{p}}}}\times \\
& \quad\times \Big\{\int_{A_{k_{h+1},\rho_h}} \!\!\!\!\!\!\mu(x)\big(
(u-k_{h+1})^q+k_{h+1}^\gamma\big)\,dx+|A_{k_{h+1},\rho_h}|^{1/s'}\Big\}^{\frac{qs'}{\overline{\boldsymbol{p}}}} \\
& \leq c(\mathbf{data})|A_{k_{h+1},{\rho}_{h+1}}|^{1-\frac{qs'}{\overline{\Bdsigma}^{\ast}}}
\Big(\frac{2^h}{R}\Big)^{\frac{q^2s'}{\overline{\boldsymbol{p}}}} \mathcal{R}^{\frac{q}{\overline{\boldsymbol{p}}}},
 \end{split}
\end{equation}
where we have introduced the shorthand notation 
$$\mathcal{R} := \int_{A_{k_{h+1},\rho_h}} \!\!\!\!\!\!\big(
	(u-k_{h+1})^q+k_{h+1}^\gamma\big)^{s'}\,dx+|A_{k_{h+1},\rho_h}|.$$
With estimate \eqref{eq:daquicomeBCM} at hand, the final step
towards the proof
of the claimed \eqref{stimaJh} consists
of estimating the above term
$\mathcal{R}$. 

To this end we first observe that, by definition
$J_h$, we have
\begin{equation} \label{eq.estimJhgeq}
   \begin{split}
    J_h & = \int_{A_{k_{h},\rho_h}}(u-k_h)^{qs'}\,d x
    \geq \int_{A_{k_{h+1},\rho_h}}(u-k_h)^{qs'}\,d x \\
    & \geq (k_{h+1}-k_h)^{qs'}\,\big|A_{k_{h+1},\rho_h}\big|
    = \bigg(\frac{d}{2^{h+2}}\bigg)^{qs'}\,\big|A_{k_{h+1},\rho_h}\big|\\
    & \geq \bigg(\frac{d}{2^{h+2}}\bigg)^{qs'}\,\big|A_{k_{h+1},\rho_{h+1}}\big|.
   \end{split}
   \end{equation}
Moreover, taking into account that
$$k_{h+1}^{\overline{\Bdsigma}^*}| A_{k_{h+1},\rho_h}|\leq \int_{A_{k_{h+1},\rho_h}}
|u|^{\overline{\Bdsigma}^*}\,  
  dx\le \|u\|_{L^{\overline{\Bdsigma}^*}(B_R(x_0))}^{\overline{\Bdsigma}^*},$$ 
  we immediately derive
  (remind that $k_h\leq d$ for any $h\geq 0$)
   \begin{align*}
    \int_{A_{k_{h+1},\rho_h}}k_{h+1}^{\gamma s'}\,
    d x&=\left(k_{h+1}^{\overline{\Bdsigma}^*}| 
    A_{k_{h+1},\rho_h}|\right)^{\frac{(\gamma-q)s'}{\overline{\Bdsigma}^*}}
    k_{h+1}^{qs'}| A_{k_{h+1},\rho_h}|^{1-\frac{(\gamma-q)s'}{\overline{\Bdsigma}^*}}
   \\ 
   & \le 
    \|u\|_{L^{\overline{\Bdsigma}^*}(B_R(x_0))}^{(\gamma-q)s'}
      d^{qs'}| A_{k_{h+1},\rho_h}|^{1-\frac{(\gamma-q)s'}{\overline{\Bdsigma}^*}}.
      \end{align*}
As a consequence, reminding \eqref{e:Jh+1leJhGC}, we obtain 
 \begin{equation} \label{eq:daBCMGrezza}
 \begin{split}
& \int_{A_{k_{h+1},\rho_h}}
    \Big((u-k_{h+1})^q+k_{h+1}^\gamma\Big)^{s'}\,d x \\
      & \qquad \leq  
      c(\mathbf{data}) \Big(J_h
      + \|u\|_{L^{\overline{\Bdsigma}^*}(B_R(x_0))}^{(\gamma-q)s'}
      d^{qs'}| A_{k_{h+1},\rho_h}|^{\theta}\Big),
 \end{split}
 \end{equation}
 where we have introduced the shorthand notation
 $$\theta = 1-\frac{(\gamma-q)s'}{\overline{\Bdsigma}^*}\in(0,1).$$
 Summing up, by combining \eqref{eq.estimJhgeq}\,-\,\eqref{eq:daBCMGrezza},
 and by taking into acco\-unt \eqref{eq.Jhleqone},
 we then deduce the following estimate for $\mathcal{R}$:
\begin{equation} \label{eq:daBCMI}
\begin{split}
 & \mathcal{R}  \leq c(\mathbf{data})\big[1+\|u\|^{(\gamma-q)s'}_{L^{\overline{\Bdsigma}^*}(B_R(x_0))}\big]
\cdot 2^{hqs'}d^{\frac{q(\gamma-q)(s')^2}{\overline{\Bdsigma}^*}}
\big\{J_h+J_h^{\theta}\big\} \\
&\qquad
\leq c(\mathbf{data})\big[1+\|u\|^{(\gamma-q)s'}_{L^{\overline{\Bdsigma}^*}(B_R(x_0))}\big]
\cdot 2^{hqs'}d^{\frac{q(\gamma-q)(s')^2}{\overline{\Bdsigma}^*}}\times \\
&\qquad\qquad\times (1+ \|u\|_{L^{qs'}(B_R(x_0))}^{qs'(1-\theta)})J_h^\theta\\
& \qquad (\text{by H\"older's inequality
with exponent $\overline{\Bdsigma}^*/(qs') > 1$}) \\
& \qquad\leq 
c(\mathbf{data})\big[1+\|u\|_{L^{\overline{\Bdsigma}^*}(B_R(x_0))}\big]^{(\gamma-q)s'\big(1+\frac{qs'}{\overline{\Bdsigma}^*}\big)}
\times \\
& \qquad\qquad\times  2^{hqs'}d^{\frac{q(\gamma-q)(s')^2}{\overline{\Bdsigma}^*}}J_h^\theta,
    \end{split}
\end{equation}
where we have also used the fact that $R_0\leq 1$.
\vspace{0.1cm}

Thanks to
estimate \eqref{eq:daBCMI}, we can now complete
the proof of \eqref{stimaJh}. Inde\-ed,
 ga\-the\-ring \eqref{eq.estimJhgeq}\,-\,\eqref{eq:daBCMI}, from \eqref{eq:daquicomeBCM}
 we obtain
 \begin{align*}
  J_{h+1}
  &\leq c(\mathbf{data})\Big(\frac{2^h}{d}\Big)^{qs'(1-\frac{qs'}{\overline{\Bdsigma}^*})}
  J_h^{1-\frac{qs'}{\overline{\Bdsigma}^*}}
  \cdot\Big(\frac{2^h}{R}\Big)^{\frac{q^2s'}{\overline{\boldsymbol{p}}}}\times
  \\
  &\qquad\times
  \Big\{\big[1+\|u\|_{L^{\overline{\Bdsigma}^*}(B_R(x_0))}\big]^{(\gamma-q)s'\big(1+\frac{qs'}{\overline{\Bdsigma}^*}\big)}\cdot 
   2^{hq s'} d^{\frac{q(\gamma-q)(s')^2}{\overline{\Bdsigma}^*}} 
   J_h^\theta\Big\}^{\frac{q}{\overline{\boldsymbol{p}}}} \\
   & \leq c(\mathbf{data})\big[1+
    \|u\|_{L^{\overline{\Bdsigma}^*}(B_R(x_0))}\big]^{
    		\frac{q}{\overline{\boldsymbol{p}}}(\gamma-q)s'\big(1+\frac{qs'}{\overline{\Bdsigma}^*}\big)}
    \cdot\frac{1}{d^{\delta_1}R^{\delta_2}}\cdot \lambda^h J_{h}^{1+\alpha},
 \end{align*}
 where the constants $\lambda,\delta_1,\delta_2$ and $\alpha$ are explicitly
 given by
 \begin{align*}
  (\ast)\,\,&\lambda = 8^{\frac{q^2 s'}{\overline{\boldsymbol{p}}}};
  \\[0.1cm]
  (\ast)  \,\,&\delta_1 =qs'\Big(1-\frac{qs'}{\overline{\Bdsigma}^*}\Big)-
  \frac{(\gamma-q)(qs')^2}{\overline{\boldsymbol{p}}\cdot\overline{\Bdsigma}^*}
  =qs'\left\{
  	1-\frac{qs'}{\overline{\Bdsigma}^*} + \frac{q^2s'}{\overline{\boldsymbol{p}}\cdot\overline{\Bdsigma}^*}- \frac{\gamma qs'}{\overline{\boldsymbol{p}}\cdot\overline{\Bdsigma}^*}\right\};
  \\[0.1cm]
  (\ast)\,\,&\delta_2 = \frac{q^2s'}{\overline{\boldsymbol{p}}};\\[0.1cm]
  (\ast)\,\,&\alpha = \frac{\theta q}{\overline{\boldsymbol{p}}}-\frac{qs'}{\overline{\Bdsigma}^*}
  = \frac{q}{\overline{\boldsymbol{p}}}\left\{1 +
\left( q-\overline{\boldsymbol{p}}\right)
  \frac{s'}{\overline{\Bdsigma}^*}-\frac{\gamma s'}{\overline{\Bdsigma}^*}\right\}.
 \end{align*}
 We finally observe that, since $s'\geq 1$ and since $q\geq\overline{\boldsymbol{p}}$ (see
 Lemma \ref{integrabilit-u}), we clearly have that $\delta_2 > 0$ and $\lambda> 1$;
 moreover, by using assumption \eqref{eq:assumptionsParameters} it is not difficult to recognize that
 $\delta_1,\alpha> 0$. 
 
 Indeed,  as regards $\delta_1$ we have
 \begin{equation} \label{e:delta1} 
 \begin{split}
 \delta_1>0& \,\,\Longleftrightarrow\,\,
  \frac{\gamma  qs'}{\overline{\boldsymbol{p}}\cdot\overline{\Bdsigma}^*}
 < 1-\frac{qs'}{\overline{\Bdsigma}^*} + \frac{q^2s'}{\overline{\boldsymbol{p}}\cdot\overline{\Bdsigma}^*}\\
 &\,\,\Longleftrightarrow\,\,\gamma<
\overline{\boldsymbol{p}}\left(
\frac{\overline{\Bdsigma}^*}{qs'} 
- 1\right)+q \\
& \,\,\Longleftrightarrow\,\,
\gamma<\frac{\overline{\boldsymbol{p}}}{q}\cdot
\frac{\overline{\Bdsigma}^*}{s'}+q- 
\overline{\boldsymbol{p}},
 \end{split}	
 \end{equation}
and this last inequality is \emph{precisely} the third inequality in \eqref{eq:assumptionsParameters}.
As regards the number $\alpha$, instead,  we notice that 
\[\alpha>0\,\,\Longleftrightarrow\,\, \frac{\gamma s'}{\overline{\Bdsigma}^*}<
1+
\left( q-\overline{\boldsymbol{p}}\right)
\frac{s'}{\overline{\Bdsigma}^*}\,\,\Longleftrightarrow\,\,\gamma <\frac{\overline{\Bdsigma}^*}{s'}+
 q-\overline{\boldsymbol{p}}
\]
and this condition is implied by \eqref{e:delta1} since $\frac{\overline{\boldsymbol{p}}}{q}\le 1$. 
It is worth highlight\-ing that, if $p_i=p=\gamma=q$ and $r_i=s=\infty$, then
 $$\delta_1=p\big(1-\frac{p}{p^*}\big)=\frac{p^2}{n}\quad\text{and}\quad 
 \alpha=1-\frac{p}{p^*}=\frac{p}{n}.$$ 
 Summing up, the claimed \eqref{stimaJh} is fully established.
 \medskip
 
 Thanks to the estimated \eqref{stimaJh} just established,
 we are finally in a position to com\-plete the demonstration of the theorem.
 Indeed, setting
 $$A = \mathbf{C}\big[1+
    \|u\|_{L^{\overline{\Bdsigma}^*}(B_{R}(x_0))})\big]^{
    	\frac{q}{\overline{\boldsymbol{p}}}(\gamma-q)s'\big(1+\frac{qs'}{\overline{\Bdsigma}^*}\big)}
    \cdot\frac{1}{d^{\delta_1}R^{\delta_2}},$$
    we can clearly rewrite the cited estimate \eqref{stimaJh} as follows
     $$J_{h+1}\leq A\,\lambda^h J_h^{1+\alpha}.$$
 Moreover, taking into account \eqref{eq.Jhleqone} (and recalling that
 $\delta_1 > 0$), we easily see that we can choose $d\geq 2$ in such a way that
 \begin{align*}
  A^{-\frac{1}{\alpha}}\lambda^{-\frac{1}{\alpha^2}}
  & = \left(\frac{d^{\delta_1}R^{\delta_2}}{ \mathbf{C}\big[1+
  	\|u\|_{L^{\overline{\Bdsigma}^*}(B_{R}(x_0))}\big]^{
  		\frac{q}{\overline{\boldsymbol{p}}}(\gamma-q)s'\big(1+\frac{qs'}{\overline{\Bdsigma}^*}\big)}}\right)^{1/\alpha}\lambda^{-1/\alpha^2}
    \\
    & \geq \|u\|_{L^{qs'}(B_R(x_0))}^{qs'}\geq J_0.
  \end{align*}
  Indeed, by H\"older's inequality (and since $R\leq R_0\leq 1$) we have  \[ \|u\|_{L^{qs'}(B_R(x_0))}^{qs'}\le
   \|u\|_{L^{\overline{\Bdsigma}^*}(B_R(x_0))}^{qs'}|B_R|^{1-\frac{qs'}{\overline{\Bdsigma}^*}}\le c_0
   \big[1+
  \|u\|_{L^{\overline{\Bdsigma}^*}(B_{R}(x_0))}\big]^{qs'}\]
  where $c_0 > 0$ is a positive constant only depending  on $n,q,s,
  \boldsymbol{\sigma}$; as a consequence of this estimate, it 
  then suffices to choose 
  \begin{equation} \label{eq:choiced}
   d= \Big\{\frac{\mathbf{C}c_0^{\alpha} \lambda^{1/\alpha}}{R^{\delta_2}}\big[1+
    \|u\|_{L^{\overline{\Bdsigma}^*}(B_{R}(x_0))}\big
    ]^{
   	\frac{q}{\overline{\boldsymbol{p}}}(\gamma-q)s'\big(1+\frac{qs'}{\overline{\Bdsigma}^*}\big)+qs'\alpha}\Big\}^{1/\delta_1}\geq 2
   \end{equation}
   (by possibly enlarging $\mathbf{C}$ if needed, and since $R_0\leq 1<\lambda$).
   \vspace{0.1cm}
   
    Thus, by exploiting Lemma \ref{lemma-giusti2} \emph{with this choice of $d$},
    we get
    $$\lim_{h\to +\infty}J_h=\int_{A_{d,R/2}}|u-d|^{qs'}\,dx=0,$$
    from which we derive that
 $|A_{d,R/2}|=0$, that is,
 $$\text{$u\leq d$ a.e.\,in $B_{R/2}(x_0)$}.$$
To prove that $u$ is also bounded from below by $-d$ it suffices to observe that,
   setting $v := -u$, then $v\in W^{1,1}_{\mathrm{loc}}(\Omega)$ is a quasi-minimizer
   of the functional
   $${G}(v) := \int_{\Omega}g(x,v,Dv)\,d x,$$
   where $g:\Omega\times\R\times\mathbb{R}^n\to\R$ is the function defined by
   $$g(x,w,\xi) := f(x,-w,-\xi).$$
   As a consequence, since $g$ clearly 
   satisfies assumptions $\mathbf{(H1)}$\,-\,$\mathbf{(H2)}$ (with the \emph{same data}), 
   by the above argument we deduce that
   $$\text{$-u \leq d$ for a.e.\,$x\in B_{R/2}(x_0)$}.$$
   Summing up, we have proved that, if $d\geq 2$ is as in \eqref{eq:choiced}, one has
   \begin{equation}  \label{eq.ulocallybounded}
    \|u\|_{L^\infty(B_{R/2}(x_0))}\leq d.
   \end{equation}
   We finally observe that, owing to the explicit expression of $\alpha$, we have
   \begin{align*}
    & \frac{q}{\overline{\boldsymbol{p}}}(\gamma-q)s'\Big(1+\frac{qs'}
    {\overline{\boldsymbol{\sigma}}^*}\Big)+qs'\alpha \\
    & \qquad 
    = \frac{qs'}{\overline{\boldsymbol{p}}}\left\{(\gamma-q)\Big(1+\frac{qs'}
    {\overline{\boldsymbol{\sigma}}^*}\Big)+q\Big[1 +
\left( q-\overline{\boldsymbol{p}}\right)
  \frac{s'}{\overline{\Bdsigma}^*}-\frac{\gamma s'}{\overline{\Bdsigma}^*}\Big]\right\} \\
  & \qquad = \frac{qs'}{\overline{\boldsymbol{p}}}\Big(\gamma-\frac{qs'\,\overline{\boldsymbol{p}}}{\overline{\Bdsigma}^*}\Big);
     \end{align*}
   as a consequence, by definition of $\delta_1,\delta_2$ we obtain
   $$d\leq c(\mathbf{data})\frac{1}{R^{\frac{q\overline{\boldsymbol{\Bdsigma}}^*}
     {\overline{\boldsymbol{p}}\cdot\overline{\boldsymbol{\Bdsigma}}^*-{qs'}({\gamma-q}+
     \overline{\boldsymbol{p}})}}}
     \big[1+
    \|u\|_{L^{\overline{\Bdsigma}^*}(B_{R}(x_0))}\big]
    ^{\frac{\overline{\boldsymbol{\sigma}}^*\gamma-qs'\overline{\boldsymbol{p}}}{\overline{\boldsymbol{p}}\cdot\overline{\boldsymbol{\sigma}}^*-qs'(\gamma-q+\overline{\boldsymbol{p}})}}.$$
     This, together with \eqref{eq.ulocallybounded}, 
     gives the desired \eqref{stima-teorema}.
\end{proof}

%
%

\end{document}